\documentclass[12pt]{article}

\usepackage[a4paper,margin=1.02in]{geometry}
\usepackage{amsmath,amssymb,amsthm,mathtools}
\usepackage{enumitem}
\usepackage{microtype}
\usepackage{booktabs}
\usepackage{xcolor}
\usepackage[colorlinks=true,linkcolor=blue!45!black,citecolor=blue!45!black,urlcolor=blue!45!black]{hyperref}

\definecolor{darkblue}{RGB}{0, 0, 100}
\definecolor{darkorange}{RGB}{200, 100, 0}

\setlist{nosep}
\allowdisplaybreaks

\newtheorem{theorem}{Theorem}[section]
\newtheorem{lemma}[theorem]{Lemma}

\newtheorem{corollary}[theorem]{Corollary}

\newtheorem{fact}[theorem]{Fact}
\theoremstyle{definition}
\newtheorem{definition}[theorem]{Definition}
\newtheorem{problem}[theorem]{Problem}
\newtheorem{question}[theorem]{Question}
\theoremstyle{remark}

\newcommand{\eps}{\varepsilon}

\newcommand{\one}{\mathbf 1}
\newcommand{\ei}{\mathbf e}

\newcommand{\ad}{\mathrm{ad}}

\newcommand{\sigore}{\sigma_{\mathrm{Ore}}^*}

\newcommand{\ds}{d^{*}}

\newcommand{\Cvec}{\widetilde C}

\title{Dominant-Degree Conditions for Ramsey--Tur\'an Factors of Non-Directed Cycle Orientations}
\author{Jia Zhou$^1$, Yunshu Gao$^1$\thanks{{Corresponding author. E-mail: gysh2004@gmail.com}}\\
\small $^1$School of Mathematics and Statistics, Ningxia University, Yinchuan 750021, China}
\date{}

\begin{document}
\maketitle

\begin{abstract}
Let $\Cvec$ be a fixed orientation of the cycle $C_\ell$, $\ell\ge3$, which is not directed.  For an oriented graph $D$,  let
 $d_D^*(v):=\max\{d_D^+(v),d_D^-(v)\},$ and let
\[
\sigore(D):=\min\bigl\{d_D^*(x)+d_D^*(y):x\ne y,\ xy,yx\notin A(D)\bigr\},
\]
with $\sigore(D)=\infty$ if the underlying graph of $D$ is complete.
We prove that, for every $\mu>0$, there exist $\gamma>0$ and $n_0$ such that every $n\ge n_0$ with $\ell\mid n$ and every $n$-vertex oriented graph $D$ satisfying
\[
\alpha(D)\le\gamma n \text{ and }
 {\sigore(D)\ge\left(\frac34+\mu\right)n}
\]
contains a $\Cvec$-factor.  {Additionally, for every fixed $s\ge2$ and every fixed real constant $C$, we construct arbitrarily large oriented graphs with $\sigore(D)\ge \frac34n+C$ that contain no $C_{2s}^{\ad}$-factor. More precisely, $C:=\frac34\alpha(D)-2$ for $s=2$ and $C:=\frac14\alpha(D)-\frac32$ for $s\ge3$.} This paper develops a weighted reduction framework adapted to dominant degree condition, proves the absorption lemma via closed-cluster merging with even-walk, and derives almost-perfect tiling structures by virtue of Farkas-lemma-based fractional decomposition.
\end{abstract}

\section{Introduction}
Perfect tiling problems constitute a core topic in extremal graph theory. For a fixed graph or digraph $H$, an \textbf{$H$-factor} is a set of vertex-disjoint $H$-copies covering the entire vertex set of the graph. A central problem in this field is to find the optimal density conditions that guarantee the existence of an $H$-factor. The classical Hajnal--Szemer\'edi theorem determines the tight minimum-degree threshold for perfect $K_r$-tilings. Subsequent studies have extensively enriched the general theory of $H$-factors; {see \cite{AlonYuster,CHWY2024,CorradiHajnal,HajnalSzemeredi,KuhnOsthus}.}

A \textbf{transitive tournament} $TT_r$ is the unique acyclic orientation of the complete graph $K_r$. Several sufficient degree conditions for $TT_r$-factors have been established; {see \cite{BaloghLoMolla,DeBiasioLoMollaTreglown,Treglown,Yuster2003}.}
In contrast, a general \textbf{tournament} $T_r$ is an arbitrary orientation of $K_r$. For any fixed $T_r$, Treglown \cite{Treglown} proved that all sufficiently large $n$-vertex digraphs with minimum semidegree at least $ (1-1/r)n$ admit a $T_r$-factor. This minimum semidegree is tight. Treglown further introduced the \textbf{dominant degree} $d^*_D(v)=\max\{d_D^+(v),d_D^-(v)\}$ for each vertex $v$ and established a sufficient dominant-degree condition for $TT_r$-factors, while conjecturing the tight threshold $d^*_D(v)\ge (1-1/r)n$. This conjecture was recently confirmed by Chang, Wei and Yan \cite{ChangWeiYan2026}, who proved the following stronger Ore-dominant-degree statement.
\begin{theorem} \cite{ChangWeiYan2026}
Let $n,r\in\mathbb{N}$ with $r\mid n$ and $r\ge 2$. Every $n$-vertex digraph $D$ satisfying
\[
d_D^*(x)+d_D^*(y)\ge 2\left(1-\frac1r\right)n-1
\]
for every  $xy\notin A(D)$ contains a $TT_r$-factor.
\end{theorem}

As a cornerstone of factor theory, Dirac's theorem has motivated extensive research on Hamiltonian cycles and cycle-factor problems in digraphs; see, for example, \cite{book,Kelly2011,KellyKuhnOsthus}. Exploring degree conditions that guarantee $\vec{C}$-factors for arbitrarily oriented cycles is a natural generalization of classical cycle-factor problems. {Czygrinow, Kierstead and Molla \cite{CKM2014} developed a directed Corr\'adi--Hajnal theory for triangle factors under minimum total-degree conditions; in particular, their results address cyclic and transitive triangles.} For odd oriented cycles, Molla and Treglown \cite{MollaTreglown2026} established the asymptotically optimal semidegree threshold for $\vec{C}$-factors.

An \textbf{oriented graph} is a digraph obtained by assigning a direction to each edge of a  graph. The thresholds in the above results can often be lowered when restricted to oriented graphs. In 2025, Lo proved the following.
\begin{theorem} {\cite{Lo2025}}
For all $\varepsilon>0$, there exists $\ell_0=\ell_0(\varepsilon)$ such that for all $\ell\ge \ell_0$ and any oriented cycle $C_\ell$ on $\ell$ vertices, any oriented graph $G$ on $k\ell$ vertices with minimum semidegree $\delta^0(G)\ge (3/8+\varepsilon)k\ell$ contains a $C_\ell$-factor. Moreover, the constant $3/8$ cannot be improved.
\end{theorem}

 A typical example of a non-directed cycle orientation is the \textbf{anti-directed cycle} $C_{2s}^{\mathrm{ad}}$, an even cycle with alternating edge directions. Chen \cite{ChenAntiDirected} showed that a minimum semidegree of at least $\bigl(\tfrac13+o(1)\bigr)n$ guarantees a $C_{2s}^{\mathrm{ad}}$-factor in oriented graphs.

Ramsey--Tur\'an theory offers a powerful framework for $H$-factor problems under independence number constraints.
 The \textbf{independence number} $\alpha(G)$ of a graph $G$ is defined as the maximum cardinality of an edgeless vertex subset.  {For a digraph $D$, $\alpha(D)$ always denotes the independence number of its underlying undirected graph.} Initiated by Erd\H{o}s and S\'os \cite{ErdosSos}, Ramsey--Tur\'an type problems have been widely studied \cite{CHY2025,ErdosHajnalSosSzemeredi,HMWY2024,HNY2026,SimonovitsSos}. Balogh, Molla, and Sharifzadeh \cite{BaloghMollaSharifzadeh} first proposed the Ramsey--Tur\'an variant of the Hajnal--Szemer\'edi theorem, which asks for degree thresholds for perfect tilings under bounded independence numbers, and resolved the case $r=3$. Knierim and Su \cite{KnierimSu} settled the problem for all $r\ge 4$, and Chen et al. \cite{CHWY2024} extended the result to general $H$-factor theory. These advances further motivate Ramsey--Tur\'an type research on digraph factors.

Chen et al. \cite{ChenKouMaZhang2026} obtained an asymptotically tight semidegree condition for $TT_3$-factors in sparse independence oriented graphs. Subsequently, Wang, Wang and Yan \cite{WWY2026} determined the optimal threshold for $TT_3$-factors and posed the following question.

\begin{question}\label{question:Wang}\cite{WWY2026}
Let $D$ be an $n$-vertex digraph with $\alpha(D)=o(n)$, and let $H$ be a digraph satisfying $|V(H)|\mid n$.
What minimum total-degree condition on $D$ guarantees that $D$ contains an $H$-factor?
\end{question}


In this paper, we study Ramsey--Tur\'an type problems for {arbitrary non-consistently-directed cycle orientations} using the genuine Ore parameter $\sigore(D)$ instead of the minimum semidegree condition for oriented graphs.
  Our main result is presented as follows. 

\begin{theorem}\label{thm:main}
Let $\Cvec$ be an orientation of $C_\ell$, $\ell\ge3$, which is not consistently directed.  For every $\mu>0$ there exist constants $\gamma>0$ and $n_0\in\mathbb N$ such that the following holds.  If $n\ge n_0$, $\ell\mid n$, and $D$ is an $n$-vertex oriented graph satisfying
\[
 \alpha(D)\le\gamma n
 \qquad\text{and}\qquad
 \sigore(D)\ge\left(\frac34+\mu\right)n,
\]
then $D$ contains a $\Cvec$-factor.
\end{theorem}

{The coefficient $3/4$ in Theorem~\ref{thm:main} is asymptotically best possible, Theorem~\ref{thm:no-additive-intro} shows that already for anti-directed cycles no smaller coefficient is possible, and at coefficient $3/4$ no fixed additive constant suffices.}

\begin{theorem}\label{thm:no-additive-intro}
For every fixed integer $s\ge2$ and every real constant $C$, there are arbitrarily large multiples $n$ of $2s$ and $n$-vertex oriented graphs $D$ such that
 $\sigore(D)\ge\frac34n+C,$
but $D$ has no $C_{2s}^{\ad}$-factor.  More precisely, the constructions satisfy
\[
\sigore(D)\ge
 \begin{cases}
 \frac34n+\frac34\alpha(D)-2, & s=2,\\[1mm]
 \frac34n+\frac14\alpha(D)-\frac32, & s\ge3.
 \end{cases}
\]
\end{theorem}

Our proof proceeds in four main stages and incorporates several key technical contributions. First, we construct a weighted reduced digraph adapted to the Ore-type dominating-degree condition $\sigma_{\text{ore}}(D)$, and establish a $\vec{C}$-tiling $\mathcal F_E$ covering all bad vertices (see Lemma \ref{lem:pretil}). Second, we merge closed clusters along even walks in the underlying undirected graph of the reduced digraph to construct a valid absorbing set (see Lemma \ref{lem:evenmerge} and Theorem \ref{cor:oreabsorber}). Third, we derive almost-perfect tiling structures via a Farkas-lemma-based fractional decomposition(see Theorem \ref{lem:almost}). Finally, we complete the proof by using absorption property of the constructed absorbing set.

\paragraph{Organisation.} Section 2 introduces the notation used throughout the paper and the required diregularity lemma and existing lattice-based absorption lemmas.  Based on our proof strategy,
 Section 3 thus presents the construction of Preliminary  $\vec{C}$-tiling $\mathcal F_E$, a valid absorbing set, almost-perfect tiling, and then complete the proof of Theorem \ref{thm:main}.
  Section 4 provides the lower-bound construction. This paper concludes with a concluding section.

\section{Preliminaries}\label{sec:tools}

\subsection{Basic Notation}
Notation not specified in this section is consistent with that in \cite{book}. Denote $ \boldsymbol{\mathbb N}$ for the set of positive integers. For integers $a\leq b$, let $ \boldsymbol{[a,b]}:=\{i\in\mathbb Z:a\leq i\leq b\}$, $ \boldsymbol{[b]}:=[1,b].$ If $h\in\mathbb N$, then $ \boldsymbol{h\mathbb N}$ is the set of positive
multiples of $h$. We use the notation $ \boldsymbol{0<a\ll b\ll c}$
to mean that $a$ is chosen sufficiently small in terms of $b$, and then
$b$ sufficiently small in terms of $c$.  For a set $W$ and
an integer $k\geq0$, $ \boldsymbol{\binom{W}{k}}$ denotes the family of all $k$-subsets
of $W$.    For a family $\mathcal{F}$ of sets, $ \boldsymbol{V(\mathcal{F})}$ denotes the union of the vertex sets of all members of $\mathcal{F}$.
  Given sets $A$ and $B$, the symbol $ \boldsymbol{A\mathbin{\cup} B}$ stands for their disjoint union, i.e., $A\cap B=\emptyset$.

All digraphs are
finite and loopless, and no multiple arcs are allowed.      For a digraph $D$, we
write $V(D)$ and $A(D)$ for its vertex set and arc set, respectively, and
write $ \boldsymbol{|D|}:=|V(D)|$ for its \textbf{order}.   If
$U\subseteq V(D)$, then $ \boldsymbol{D[U]}$ denotes the subdigraph of $D$ induced by
$U$, and $ \boldsymbol{D-U}:=D[V(D)\setminus U]$. And if $F$ is a subdigraph of $D$, we
also write $ \boldsymbol{D-F}:=D-V(F)$.
 For $v\in V(D)$, let
\[
  \boldsymbol{N_D^+(v)}:=\{x\in V(D):vx\in A(D)\},
 \qquad
  \boldsymbol{N_D^-(v)}:=\{x\in V(D):xv\in A(D)\}
\]
be the \textbf{out-neighbourhood} and \textbf{in-neighbourhood} of $v$, respectively, and
let $  \boldsymbol{d_D^+(v)}:=|N_D^+(v)|$, $
  \boldsymbol{d_D^-(v)}:=|N_D^-(v)|.$
For $X\subseteq V(D)$, put
\[
  \boldsymbol{N_D^+(v,X)}:=N_D^+(v)\cap X,
 \qquad
 \boldsymbol{ N_D^-(v,X)}:=N_D^-(v)\cap X,
\]
and use $ \boldsymbol{d_D^+(v,X)}$ and $ \boldsymbol{d_D^-(v,X)}$ for their cardinalities.  The subscript is omitted when the digraph is clear.  The \textbf{minimum
out-degree}, \textbf{minimum in-degree}, \textbf{minimum semidegree} and \textbf{minimum total degree}
of $D$ are, respectively,
\[
  \boldsymbol{\delta^+(D)}:=\min_{v\in V(D)}d_D^+(v),
 \qquad
  \boldsymbol{\delta^-(D)}:=\min_{v\in V(D)}d_D^-(v),
\]
\[
  \boldsymbol{\delta^0(D)}:=\min\{\delta^+(D),\delta^-(D)\},
 \qquad
  \boldsymbol{\delta(D)}:=\min_{v\in V(D)}\bigl(d_D^+(v)+d_D^-(v)\bigr).
\]
 We write $ \boldsymbol{TT_r}$ for the \textbf{transitive tournament} on $r$ vertices; thus
$TT_r$ has an ordering $v_1,\ldots,v_r$ in which $v_iv_j$ is an arc for
every $i<j$. A  graph $G=(V,E)$ is \textbf{connected} if for every pair of vertices $u,v\in V$, there exists a   path from $u$ to $v$ in $G$. A \textbf{connected component} of $G$ is a maximal connected subgraph of $G$. If $C$ is a connected component of $G$, then there are no edges between $V(C)$ and $V(G)\setminus V(C)$.

Throughout the paper, we employ the following conventions for the notation used in this paper.
Fix an orientation $\Cvec$ of the cycle $C_\ell$, where $\ell\ge3$, which is not a directed cycle. And fix the length  $\ell$. Then $\Cvec$  has at least one source and at least one sink; deleting either a source or a sink leaves an oriented path on $\ell-1$ vertices, which  allows the dominant-degree condition to be converted into robust copies of $\Cvec$. Let  $\kappa_\ell:=\frac{(\ell-1)(\ell-2)}2+1$. From now on, given a digraph $D$, define $P^+$ and $P^-$ to be
\begin{equation}\label{pm}
    P^+:=\{v\in V(D):d^+(v)\ge d^-(v)\}
 \qquad \text{and}\qquad  P^-:=V(D)\setminus P^+,
\end{equation}
 {Either of $P^+$ and $P^-$ is allowed to be empty.}



\subsection{Tools}
Our proof relies mainly on the lattice absorption method, so we now present the basic definitions and key auxiliary lemmas needed for the argument.
\begin{definition}\label{def:reach}
 {Let $H$ be a fixed oriented graph on $h$ vertices. Distinct vertices $x,y\in V(D)$ are \textbf{$(H,m,t)$-reachable} if for every set
$W\subseteq V(D)\setminus\{x,y\}$ with $|W|\le m$, there is a set $Q\subseteq V(D)\setminus(W\cup\{x,y\})$ with $|Q|\le th-1$ such that both $D[Q\cup\{x\}]$ and $D[Q\cup\{y\}]$ have $H$-factors. A set $U$ is \textbf{$(H,m,t)$-closed} if every pair of distinct vertices in $U$ is $(H,m,t)$-reachable. In all applications below we take $H=\Cvec$ and hence $h=\ell$.}
\end{definition}

For a partition $\mathcal P=(U_1,\ldots,U_p)$ and $S\subseteq V(D)$, set
\[
 \mathbf i_{\mathcal P}(S):=(|S\cap U_1|,\ldots,|S\cap U_p|).
\]
 {For a fixed oriented graph $H$ on $h$ vertices, a $p$-vector $\mathbf v$ is \textbf{$(H,\rho)$-robust} for $\mathcal P$ if for every $W\subseteq V(D)$ with $|W|\le\rho n$ there is an $H$-copy $F\subseteq D-W$ such that $\mathbf i_{\mathcal P}(V(F))=\mathbf v$.}


\begin{lemma}[Robust profile lemma]\label{lem:transferral}

Fix $\ell\ge3$ and a non-directed orientation $\Cvec$ of $C_\ell$. For every $d>0$, $M\in\mathbb N$, and $0<\varepsilon\le d/8$, there exist constants $\nu,\gamma>0$ and $n_0\in\mathbb N$ such that the following holds. Let $D$ be an $n$-vertex oriented graph, $n\ge n_0$, with a partition $\mathcal P=(V_1,\dots,V_k)$. Suppose that for distinct $i,j\in[k]$ there are $X\subseteq V_i$ and $Y\subseteq V_j$ with $|X|,|Y|\ge n/M$ such that $(X,Y)$ is an $(\varepsilon , d)$-regular directed pair. If $\alpha(D)\le\gamma n$, then both
\[
\mathbf e_i+(\ell-1)\mathbf e_j\quad\text{and}\quad (\ell-1)\mathbf e_i+\mathbf e_j
\]
are $(\Cvec,\nu)$-robust with respect to $\mathcal P$.

\end{lemma}
\begin{proof}
Choose a source $s$ and a sink $t$ of $\Cvec$, and recall that
$\kappa_\ell:=(\ell-2)^2+1$.  Choose $\nu$ and $\gamma$, in the order, such that $0<\nu\ll\min\{\varepsilon/(2M),d/(8M)\} $ and $ 0<\gamma\ll d/(4MK).$  Fix an arbitrary set $W\subseteq V(D)$ with
$|W|\le\nu n$.

Because $(X,Y)$ is an $(\varepsilon , d)$-regular directed pair, it follows that all but at most $\varepsilon m$ vertices $x\in X$ satisfy
$d_D^+(x,Y)\ge(d-\varepsilon)m$.  Since $\nu n\le\varepsilon m/2$, there is a vertex
$x\in X\setminus W$ with this property.  Hence
\[
|N_D^+(x)\cap(Y\setminus W)|
\ge(d-\varepsilon)m-\nu n
\ge \frac d2m
\ge \frac d{2M}n.
\]
Let $S:=D[N_D^+(x)\cap(Y\setminus W)]$.  Since
$\alpha(S)\le\alpha(D)\le\gamma n$, we have
\[
\chi(UG(S))\ge \frac{|S|}{\alpha(S)}
\ge \frac{dn/(2M)}{\gamma n}
\ge 2K>K.
\]
By the embedding lemma, $S$ contains a copy of the oriented path $\Cvec-s$; adding $x$ gives a copy of $\Cvec$ with index vector $\mathbf e_i+(\ell-1)\mathbf e_j$.

Similarly, all but at most $\varepsilon m$ vertices $y\in Y$ satisfy
$d_D^-(y,X)\ge(d-\varepsilon)m$.  Choose such a vertex $y\in Y\setminus W$ and set
$S':=N_D^-(y)\cap(X\setminus W)$.  Then
$|S'|\ge d m/2\ge dn/(2M)$, and the same chromatic estimate gives
$\chi(UG(D[S']))>K$.  Thus $D[S']$ contains a copy of $\Cvec-t$; adding $y$ gives a copy of $\Cvec$ with index vector $(\ell-1)\mathbf e_i+\mathbf e_j$. Since $W$ was arbitrary, both index vectors are $(\Cvec,\nu)$-robust.
\end{proof}

\begin{lemma}[Oriented transferral merging]\label{lem:merge}{\cite[Lemma~5.11]{CHWY2024}}
 {Let $H$ be a fixed oriented graph on $h\ge3$ vertices, let $L\in\mathbb N$, and let $\beta>0$. For all sufficiently large $N$, suppose that an oriented graph $G$ of order $N$ has a partition $\mathcal P=\{V_1,\dots,V_p\}$ such that every $V_i$ is $(H,\beta N,L)$-closed. If, for distinct $i,j$, there are two $(H,\beta)$-robust $h$-vectors $\mathbf p,\mathbf q$ satisfying $\mathbf p-\mathbf q=\mathbf e_i-\mathbf e_j$, then $V_i\cup V_j$ is $(H,\beta N/2,L')$-closed for an integer $L'=L'(h,L)$.}
\end{lemma}

\begin{lemma}[Absorbing lemma]\label{lem:absorber} {\cite{LoMarkstrom,CHWY2024}}
Let $H$ be a fixed oriented graph on $h\ge2$ vertices. For every $\eta,\beta>0$ and $t\in\mathbb N$ there are $\xi>0$ and $n_0$ such that the following holds. If $D$ is a digraph with $n\ge n_0$, and $V(D)$ is $(H,\beta n,t)$-closed, then there is a set $A\subseteq V(D)$ with $|A|\le\eta n$ such that, for any set $U\subseteq V(D)\setminus A$ with $|U|\le\xi n$ and $h\mid |A\cup U|$, the digraph $D[A\cup U]$ has an $H$-factor.
\end{lemma}

We shall frequently use the elementary bound
\begin{fact}\label{fact1}
Every graph $G$ satisfies $\chi(G)\ge |V(G)|/\alpha(G)$.
\end{fact}

\begin{lemma}[Embedding Lemma]\label{lem:tree}\cite{AddarioBerry2013}
Every $(k^2/2 - k/2 + 1)$-chromatic digraph contains every oriented tree of order $k$.
\end{lemma}

The lattice-absorption method consists of three key steps. First, we construct a small absorbing set using degree conditions together with robust index-vector constraints, which is guaranteed by the absorbing lemma. Second, we establish an almost-target structure on the residual graph such that only a small number of vertices remain uncovered, and this is achieved via the almost-covering lemma. Third, we absorb those uncovered vertices into the absorbing set to produce the desired global configuration. To prove the desired absorbing lemma and almost-covering lemma, we need a powerful tool, namely the Diregularity Lemma, which is a version of the Regularity Lemma for digraphs due to Alon and Shapira \cite{AlonShapira}.

The \textbf{density} of a bipartite graph $G=(A,B)$ with vertex classes $A$ and $B$ is defined to be
\[
\boldsymbol{d_G(A,B)}:=\frac{e_G(A,B)}{|A||B|}.
\]
We often write $d(A,B)$ if this is unambiguous. Given $\varepsilon>0$, we say that $G$ is \textbf{$\boldsymbol{\varepsilon}$-regular} if for all subsets $X\subseteq A$ and $Y\subseteq B$ with $|X|>\varepsilon|A|$ and $|Y|>\varepsilon|B|$ we have that $|d(X,Y)-d(A,B)|<\varepsilon$. Similarly, a directed pair $(X,Y)$ is \textbf{$\boldsymbol{\varepsilon}$-regular} if
for all $X'\subseteq X$ and $Y'\subseteq Y$ with
$|X'|\geq\varepsilon|X|$ and $|Y'|\geq\varepsilon|Y|$, we have $|d_D(X',Y')-d_D(X,Y)|<\varepsilon.$ It is \textbf{$\boldsymbol{(\varepsilon,d)}$-regular} if it is $\varepsilon$-regular and
$d_D(X,Y)\geq d$.  We shall repeatedly use the following standard slicing consequences: if $(X,Y)$ is $\eps$-regular of density $p$ and $Y'\subseteq Y$ has $|Y'|\ge\eps |Y|$, then
\begin{equation}\label{2.1}
    \text{all but at most $\eps|X|$ vertices $x\in X$ satisfy } d_D^+(x,Y')\ge (p-\eps)|Y'|.
\end{equation}
The analogous statement with in-neighbours is immediate by reversing all arcs.

\begin{lemma}[Koml\'os and Simonovits~\cite{KomlosSimonovits}, Slicing Lemma]\label{lem:slicing}
Assume $(V_1,V_2)$ is $\varepsilon$-regular with density $\beta$. For some $\alpha\ge \varepsilon$, let $V_1'\subseteq V_1$ with $|V_1'|\ge \alpha|V_1|$ and $V_2'\subseteq V_2$ with $|V_2'|\ge \alpha|V_2|$. Then $(V_1',V_2')$ is $\varepsilon'$-regular with $\varepsilon':=\max\{2\varepsilon,\varepsilon/\alpha\}$ and for its density $\beta'$ we have $|\beta'-\beta|<\varepsilon$.
\end{lemma}

We employ the degree form of the directed regularity lemma due to Alon and Shapira~\cite{AlonShapira}, which is stated below.

\begin{lemma}[Degree-form diregularity]
\label{lem:regularity}
For every real number $0<\eps<1$, every real number $0\le d<1$, and all integers $M_0,q\ge1$, there exist integers
$M=M(\eps,d,M_0,q)\ge M_0$ and
$n_{\rm reg}=n_{\rm reg}(\eps,d,M_0,q)$
such that the following holds. Let $D$ be a digraph of order $n\ge n_{\rm reg}$, and let $\mathcal Q$ be a partition of $V(D)$ into at most $q$ nonempty parts. Then there are a partition
\[
V(D)=V_0\mathbin{\cup}V_1\mathbin{\cup}\cdots\mathbin{\cup}V_k
\]
refining $\mathcal Q$, and a spanning subdigraph $D'\subseteq D$, such that
\begin{enumerate}[label=\rm(\roman*)]
\item $M_0\le k\le M$, $|V_0|\le\eps n$, and $|V_1|=\cdots=|V_k|=:m$;
\item $D'[V_i]$ is empty for every $i\in[k]$, and $V_0$ is isolated in $D'$;
\item for every $v\in V(D)$,
$d_{D'}^+(v)\ge d_D^+(v)-(d+2\eps)n$ and $d_{D'}^-(v)\ge d_D^-(v)-(d+2\eps)n;$
\item for every distinct $i,j\in[k]$, the pair $(V_i,V_j)$ in $D'$ is $\eps$-regular and has density either $0$ or at least $d$.
\end{enumerate}
\end{lemma}

 We call the spanning digraph $D' \subseteq D$ given by the Diregularity lemma \textbf{the pure digraph}.The vertex sets $V_1,\dots,V_k$ are called \textbf{clusters}, $V_0$ is called the \textbf{exceptional set} and the vertices in $V_0$ are called \textbf{exceptional vertices.}  For convenience,   write $\boldsymbol{p_{(i,j)}(D)}:=d_{D'}(V_i,V_j)$ for each $i\ne j$.
Thus $p_{(i,j)}(D)=0$ or $p_{(i,j)}(D)\ge d$, and because $D$ is oriented, there is
 $p_{(i,j)}(D)+p_{(j,i)}(D)\le1.$ For each part $V_i$, according to the location of $V_i$, define the \textbf{dominant weight}
\[
 w_i(D):=
 \begin{cases}
  \sum_{j\ne i}p_{(i,j)}(D),&V_i\subseteq P^+,\\
  \sum_{j\ne i}p_{(j,i)}(D),&V_i\subseteq P^-.
 \end{cases}
\tag{3.4}\label{eq:wi}
\]

Let $p$ be the number of clusters $V_i$ that $V_i\subseteq P^+$ and thus $q=k-p$ is the number of clusters $V_i$ that $V_i\subseteq P^-$. The total contribution to $\sum_{i=1}^k w_i$ from a  pair $(V_i,V_j)$ with $V_i, V_j\subseteq P^+$ or $V_i, V_j\subseteq P^-$ is at most $1$, while a  pair $(V_i,V_j)$ with $V_i\subseteq P^+$, $V_j\subseteq P^-$ contributes at most $2$.  Thus
\begin{equation}\label{weightsum}
 \sum_{i=1}^k w_i
 \le \binom p2+\binom q2+2pq
 =\binom k2+pq
 \le \frac34k^2-\frac k2.
\end{equation}
Similarly,   \begin{equation}\label{weightsum-connected}
\text{for any connected component $C$ of   $UG(R)$, } \sum_{i\in C}  w_i
 \le   \frac34{|C|}^2-\frac{|C|}{2},
\end{equation} because no edge between $C$ and $UG(R)\setminus C$.

Given clusters $V_1,\dots,V_k$ and a digraph $D'$, define \textbf{the reduced digraph} $R$ with parameters $(\varepsilon,d)$ to be the digraph whose vertex set is $[k]$ and in which $ij\in A(R)$ if and only if $(V_i,V_j)$ is a directed $(\varepsilon , d)$-regular pair in $D'$. (So if $D'$ is the pure digraph, then $ij\in A(R)$  if and only if there is an arc from $V_i$ to $V_j$ in $D'$.) Therefore, the underlying graph $UG(R)$ is a graph with vertex set $[k]$, and $ij\in E(UG(R))$ if and only if $p_{(i,j)}(D)+p_{(j,i)}(D)>0$.

\subsection{Basic Properties}
In this subsection, we present some basic properties of oriented graphs under dominant-degree conditions.

\begin{lemma}\label{lem:mass}
Every $n$-vertex oriented graph $D$ satisfies
\[
 \sum_{v\in V(D)} d_D^*(v)
 \le \binom n2+\left\lfloor\frac{n^2}{4}\right\rfloor
 \le \frac34n^2-\frac12n.
\label{eq:mass}
\]
\end{lemma}

\begin{proof}
By the definition of $d^*(v)$ and \eqref{pm}, it is easy to check that $d^*(v)=d^+(v)$ for each vertex $v\in P^+$ and $d^*(v)=d^-(v)$ for each vertex $v\in P^-$. Hence {$\sum_{v\in V(D)}d^*(v)=\sum_{v\in P^+}d^+(v)+\sum_{v\in P^-}d^-(v)$.} For each arc $uv\in A(D[P^+])$, since $u\in P^+$ and $v\in P^+$, the arc $uv$ contributes $1$ through the out-degree $d^+(u)$. As we do not accumulate in-degrees for vertices in $P^+$, it yields no additional contribution from $v$. The total contribution is therefore $1$. By symmetry, any arc inside $P^-$ contributes at most $1$. Similarly, an arc from $P^+$ to $P^-$ contributes two; an arc from $P^-$ to $P^+$ contributes zero.  Hence,
\[
 \sum_v d^*(v)
 \le \binom {|P^+|}2+\binom {|P^-|}2+2{|P^+||P^-|}
 =\binom n2+{|P^+||P^-|}
 \le\binom n2+\left\lfloor\frac{n^2}{4}\right\rfloor.
\]
as desired.
\end{proof}

We then derive the following consequences.

\begin{lemma}[Ore dominant-weight lemma]\label{lem:weightedore}
Suppose that $D$ is an oriented graph of sufficiently large order $n$  with
$\sigore(D)\ge\left(\frac34+\mu\right)n$. Let $V_0,V_1,\ldots,V_k$ and $D'$ be the partition and pure digraph obtained by Lemma \ref{lem:regularity} with $D$, respectively. Let $R$ and $w_i(D)$ be the corresponding reduced digraph and dominate wight defined as above. Then the following hold.
\begin{enumerate}[label=\rm(\roman*)]
\item $\displaystyle \min_{v\in V(D)}d_D^*(v)\ge \mu n/3$;
\item $\displaystyle \min_{i\in[k]}w_i(D)\ge  \mu k/4$;
\item if $C$ is a component of $UG(R)$, $r:=|C|$, and $X_C:=\cup_{i\in C}V_i$, then
$|X_C|\ge \mu n/5$,
and the set
\[
L_C:=\left\{x\in X_C:d_D^*(x)\le\frac34|X_C|+\frac\mu8n\right\}
\]
satisfies $|L_C|\ge \mu^2n/100$.
\end{enumerate}
\end{lemma}

\begin{proof}
Let $a:=\min_{v\in V(D)}d_D^*(v)$ and choose $v$ with $d_D^*(v)=a$. Since $D$ is oriented, it follows that
$d_D^+(v)+d_D^-(v)\le2a$,
so $v$ has at least $n-1-2a$ non-neighbours in $UG(D)$. Every such non-neighbour $u$ satisfies
$d_D^*(u)\ge\left(\frac34+\mu\right)n-a$.
If $2a\ge(3/4+\mu)n$, then (i) is immediate. Otherwise,
\begin{align*}
\sum_{x\in V(D)}d_D^*(x)
&\ge na+(n-1-2a)\left(\left(\frac34+\mu\right)n-2a\right).
\end{align*}
Suppose that $a<\mu n/3$ and put $t:=a/n$, then $t<\mu/3$, we obtain
\[
\frac1{n^2}\sum_xd_D^*(x)
\ge \frac34+\mu-\left(\frac52+2\mu\right)t+4t^2-o(1)
>\frac34,
\]
where the last inequality follows from $0<\mu\le10^{-2}$ and $n$ sufficiently large. This contradicts Lemma~\ref{lem:mass}. Hence (i) holds.

For (ii), set $m=(n-|V_0|)/k$. Take $x\in V_i$ with normalized degree $\frac{d_D^{*}(x)}{n}$. Subtracting the reduced digraph density error $d$ and the regularization error $2\varepsilon$ yields the average effective contribution of $x$ towards the clusters. Since  $|D|=n$ and $|R|=k$, we obtain
\begin{equation}\label{eq:wlift}
w_i(D)\ge k\left(\frac{\min_{x\in V_i}d_D^*(x)}n-d-2\varepsilon\right).
\end{equation}
By (i) and $d,\varepsilon\ll\mu$, this yields
$w_i(D)\ge\left(\mu/3-d-2\varepsilon\right)k\ge \mu k/4$
for every $i\in [k]$.

For (iii), let $C$ be a component of $UG(R)$ of order $r$. Since there is no edge between $C$ and $UG(R)\setminus C$, we have $w_i(D)\le r$ for every vertex $i\in C$. By (ii),
$r\ge\frac\mu k/4$,
and therefore
\[
|X_C|=rm\ge\frac\mu4km=\frac\mu4(n-|V_0|)\ge\frac\mu5n.
\]
{Moreover, by the definition of $w_i(D)$,
\[
\sum_{x\in X_C}d_D^*(x)
\le m^2\sum_{i\in C}w_i(D)+|X_C|(d+2\varepsilon)n.
\]}
Using  \eqref{weightsum-connected}, and writing $h:=|X_C|$ and $\delta:=d+2\varepsilon$, we obtain
$\sum_{x\in X_C}d_D^*(x)
\le h\left(\frac34h+\delta n\right).$
Let
\[
L_C=\left\{x\in X_C:d_D^*(x)\le\frac34h+\frac\mu8n\right\}.
\]
Then
\[
(h-|L_C|)\left(\frac34h+\frac\mu8n\right)
\le h\left(\frac34h+\delta n\right),
\]
whence, since $\delta\le\mu/16$,
\[
|L_C|\ge \frac{(\mu/8-\delta)n}{(3/4)h+(\mu/8)n}h
\ge\frac\mu{16}h\ge\frac{\mu^2}{80}n>\mu^2n/100,\] as required.
\end{proof}

\begin{lemma}\label{lem:connected}
Suppose $D$ is an oriented graph of order $n\ge n_0$ satisfying
$\sigore(D)\ge\left(\frac34+\mu\right)n$ and $R$ is the reduced digraph of $D$.
 Then   $UG(R)$ is connected.
\end{lemma}

\begin{proof}
Suppose that $UG(R)$ is disconnected. Choose two distinct components $C_1,C_2$,  set
$X_t:=\cup_{i\in C_t}V_i$,  and $h_t:=|X_t|$ for each $t\in [2]$.
By Lemma~\ref{lem:weightedore}(iii), $h_t\ge\mu n/5$ and
\[
L_t:=\left\{x\in X_t:d_D^*(x)\le\frac34h_t+\frac\mu8n\right\}
\]
has size at least $\mu^2n/100$.

Since $C_1$ and $C_2$ are distinct components of $UG(R)$, the pure digraph $D'$ of $D$ contains no arc between $X_1$ and $X_2$. Hence, for every $x\in X_1$, we have
$d_D^+(x,X_2)+d_D^-(x,X_2)\le2(d+2\varepsilon)n.$
By $d,\varepsilon\ll\mu^2$,
$|L_2|>2(d+2\varepsilon)n.$
Thus, fixing any $x\in L_1$, there exists $y\in L_2$ non-adjacent to $x$. The condition $\sigore(D)\ge\left(\frac34+\mu\right)n$ gives
\[
d_D^*(x)+d_D^*(y)\ge\left(\frac34+\mu\right)n.
\]
On the other hand,
\[
d_D^*(x)+d_D^*(y)
\le\frac34(h_1+h_2)+\frac\mu4n
\le\left(\frac34+\frac\mu4\right)n,
\]
a contradiction. Hence $UG(R)$ is connected.
\end{proof}

\section{Proof of Theorem \ref{thm:main}}
In this section, we adopt the following conventions.

\noindent \textbf{Choice of parameters.}
Fix $\ell$ and $\mu$.   Choose $\eta,d, \varepsilon,M_0$ such that  $0<\eta\ll\mu$,  $0<d\ll\mu^2$, $0<\varepsilon\ll d^4,\mu^2,1/\ell$ and $\mu \ll 3< M_0$.
Apply Lemma~\ref{lem:regularity} with the initial partition  $(P^+,P^-)$  and lower bound $M_0$, and let $M$ be the resulting upper bound on the number of clusters.
Apply Lemma~\ref{lem:transferral} with parameters $(3\varepsilon,d/2,2M)$ and let $\nu$ be the resulting robustness constant.
Choose $\beta$ for Lemma~\ref{lem:clusterclosed} so that  $0<\beta\ll\nu,d^2/M$. Set $\tau:=\beta/4^{M+2}$. {Let $T=T(h,L)$ be the integer yielded by at most $M-1$ iterative applying Lemma~\ref{lem:merge}.}
Choose $\rho$ with $\beta,\tau\ll\rho\ll\nu,\mu$. {Apply Lemma~\ref{lem:absorber} with $(\eta,\tau,T)$ and thus obtain a constant  $\xi_0$; set $\xi:=\xi_0/2$.} Finally choose $\zeta\ll\xi$, choose the auxiliary almost-covering regularity constants $\varepsilon_a, d_a, M_{0,a}$ such that $1/M_{0,a}\ll \varepsilon_a\ll d_a\ll\min\{\zeta,\mu^2\}$,  and let $M_a=M(\varepsilon_a,d_a,M_{0,a},2)$ be the corresponding regularity upper bound by applying Lemma \ref{lem:regularity} with some initial partition and $(\varepsilon_a, d_a, M_{0,a})$.
Set
\[
M_a^\star:=\left\lceil\frac{2M_a}{\zeta}\right\rceil,
\qquad
\widehat\varepsilon_a:=\max\left\{2\varepsilon_a,\frac{\varepsilon_a}{\zeta}\right\},
\]
and choose $\varepsilon_a$ sufficiently small that $\widehat\varepsilon_a\le d_a/16$. Denote $\gamma_{2.2}$ to be the constant obtained by  Lemma~\ref{lem:transferral} with  $(\widehat\varepsilon_a,d_a/2,M_a^\star)$, then choose $\gamma>0$ so small that  {$2\gamma \ll \gamma_{2.2}, d^2(1-\eps)/(10M\kappa_\ell)$.} Take $n_0$ sufficiently large.

\noindent \textbf{Conventions for the proof.} Suppose $D$ is an oriented graph of order $n\ge n_0$ satisfying {$\sigore(D)\ge(\frac34+\mu)n$} and $\alpha(D)\le\gamma n$. Apply Lemma~\ref{lem:regularity} with $\mathcal Q:=(P^+,P^-)$  and parameters $\varepsilon,d,M_0$. Write
\[
V(D)=V_0\mathbin{\cup}V_1\mathbin{\cup}\cdots\mathbin{\cup}V_k,
\qquad |V_1|=\cdots=|V_k|=m,
\]
let $D'$ be the pure digraph, and let $R^0$ be its reduced digraph, and let $R:=UG(R^0)$.

\subsection{Preliminary tiling}
By Lemma~\ref{lem:weightedore}(ii), for each cluster $V_i$, there exists some cluster $V_{f(i)}$ with $f(i)\neq i$ defining a mapping $f\colon i\mapsto f(i)$, such that
\[
p_{(i,f(i))}\ge d\quad\text{whenever }V_i\subseteq P^+,\qquad
p_{(f(i),i)}\ge d\quad\text{whenever }V_i\subseteq P^-.
\]
This means that $(V_i,V_{f(i)})$ or $(V_{f(i)},V_i)$ is a directed $(\eps , d)$-regular pair. Together with \eqref{2.1}, there are some bad vertices in the pair. So define
\[
 B_i:=\begin{cases}
     \{x\in V_i:d_{D'}^+(x,V_{f(i)})<(p_{(i,f(i))}-\eps)m\}, &\text{if }V_i\subseteq P^+;\\
\{x\in V_i:d_{D'}^-(x,V_{f(i)})<(p_{(f(i),i)}-\eps)m\}, &\text{if }V_i\subseteq P^-;
 \end{cases}
\]
  By the $\eps$-regularity of pair $(V_i,V_{f(i)})$,
 $|B_i|\le\eps m$ for each $i\in [k]$.
Set $E_{bad}:=V_0\cup\bigcup_iB_i.$
Then $|E_{bad}|\le2\eps n$. For the set $E_{\mathrm{bad}}$, our goal is to find a $\widetilde{C}$-tiling $\mathcal{F}_E$ of size $|E_{\mathrm{bad}}|$ such that each $\widetilde{C}_i\in\mathcal{F}_E$ satisfies $|V(\widetilde{C}_i)\cap E_{\mathrm{bad}}|=1$ and $E_{\mathrm{bad}}\subseteq V(\mathcal{F}_E)$. We achieve this greedily by  Lemma \ref{lem:pretil}.

\begin{lemma}[Exceptional tiling]\label{lem:pretil}
The digraph $D$ has a $\Cvec$-tiling $\mathcal F_E$ satisfying
\begin{enumerate}[label=\rm(\roman*)]
\item every copy in $\mathcal F_E$ contains exactly one vertex of $E_{\rm bad}$ and $E_{\rm bad}\subseteq V(\mathcal F_E)$;
\item $|V(\mathcal F_E)|=\ell|E_{\rm bad}|\le2\ell\varepsilon n\le\sqrt\varepsilon n$;
\item {for every $i\in[k]$, $|V(\mathcal F_E)\cap V_i|\le 3\sqrt\varepsilon\,m$ for sufficiently large $n$.}
\end{enumerate}
\end{lemma}
\begin{proof}
Order $E_{\rm bad}=\{x_1,\dots,x_q\}$. Suppose $\Cvec$-copies through $x_1,\dots,x_{r-1}$ have already been chosen at step $r-1$. Let $F_{r-1}$ consist of the already used vertices together with $E_{\rm bad}\setminus\{x_r\}$. Note that $|F_r|\leq |F_q|$ for each step  $r\in [q]$ and $|F_q|\leq 2\ell\varepsilon n$. Let $\mathcal S_{r}:=\big\{V_i\,\big|\, |V_i\cap F_r|\ge \sqrt{\varepsilon} m\big\} $,  for each $r\in [q]$, and let $ t_r:=|\mathcal S_r|.$  Since $\bigcup_{V_i\in\mathcal S_{r-1}}(V_i\cap F_{r-1})\subseteq F_{r-1}$, we have $t_{r-1}\cdot \sqrt{\varepsilon} m  \le  |F_{r-1}|  <  2\ell\varepsilon n.$ Hence $t_{r-1}  <  \frac{2\ell\varepsilon n}{\sqrt{\varepsilon} m} =\frac{2\ell \sqrt{\varepsilon}  n}{m}.$ Each cluster $V_i$ has size $m$, so $ |\bigcup_{V_i\in\mathcal S_{r-1}} V_i| \;\le\; t_{r-1}\cdot m.$ Substituting the upper bound for $t_{r-1}$,   $ |\bigcup_{V_i\in\mathcal S_{r-1}} V_i | \;<\; \frac{2\ell\sqrt{\varepsilon}\,n}{m}\cdot m =2\ell\sqrt{\varepsilon}\,n.$

By Lemma~\ref{lem:weightedore}(i),
$\min_{v\in V(D)}d_D^*(v)\ge \frac\mu3n.$
If $d_D^*(x_r)=d_D^+(x_r)$, remove from $N_D^+(x_r)$ the set $F_{r-1}$ and the clusters in $\mathcal S_{r-1}$. By $\varepsilon\ll\mu^2,1/\ell^2$, {at least $\mu n/5$ vertices remain}. Their underlying graph has chromatic number at least $\kappa_\ell$, so Lemma~\ref{lem:tree} embeds $\Cvec-s$, where $s$ is a source of $\Cvec$; adding $x_r$ at $s$ gives the next copy. If $d_D^*(x_r)=d_D^-(x_r)$, use a sink and in-neighbours. Thus the copies are vertex-disjoint and each contains exactly one bad vertex. Since the vertices of $\Cvec$-copy $C_r$ at step $r$ are never selected from $\mathcal S_{r-1}$, it follows that each $V_i\in \mathcal S_{r-1}|$  satisfies $|V_i\cap F_{j}|\leq \sqrt{\eps}m+\ell+|B_i|$ for any $ j\geq r$; hence (iii) follows from $|B_i|\le\varepsilon m$ and $n$ large.
\end{proof}

Set $D_1:=D-\mathcal F_E$ and $U_i:=V_i\setminus V(\mathcal F_E)$. Lemma~\ref{lem:pretil}(iii) gives
$(1-3\sqrt\varepsilon)m\le |U_i|\le m$ for each $i\in[k]$.
Moreover, by the slicing lemma, for every arc $ij\in A(R^0)$, $(U_i,U_j)$ remains a $(3\varepsilon,d/2)$-regular directed pair, after increasing $n_0$ if necessary. Recall that $m=\frac{n-|V_0|}{k}\geq \frac{(1-\eps )n}{M}$. This gives that $ n\leq \frac{Mm}{1-\eps}$.

\subsection{Absorbing Lemma}\label{sec:closed}

\begin{lemma}\label{lem:clusterclosed}
Every $U_i$ is $(\Cvec,\beta n,2)$-closed.
\end{lemma}
\begin{proof}
We give the $P^+$ case; the $P^-$ case follows after reversing all arcs. Put $X=V_i$ and $Y=V_{f(i)}$, and write $p=p_{(i,f(i))}/3\ge d/3$. Define the auxiliary graph $J_i$ as follows.  For $x,z\in U_i$, $xz\in E(J_i)$ if and only if
\[
|N_{D'}^+(x,Y)\cap N_{D'}^+(z,Y)|\ge(p-\varepsilon)^2m.
\]
Since $U_i\cap B_i=\varnothing$, regularity implies that each $x\in U_i$ has at most $2\varepsilon m$ non-neighbours in $J_i$ before the cleaning; Lemma~\ref{lem:pretil}(iii) therefore gives at most $4\sqrt\varepsilon m$ non-neighbours in $J_i[U_i]$ for $n$ large.

{First let $W\subseteq V(D_1)$ with $|W|\le2\beta n$.} If $xz\in E(J_i)$ and $x,z\notin W$, then
\[
\bigl|(N_{D'}^+(x,Y)\cap N_{D'}^+(z,Y))\setminus(V(\mathcal F_E)\cup W)\bigr|\ge [(p-\varepsilon)^2-3\sqrt\varepsilon-3\beta M]m
\]
this is because that $|V(\mathcal F_E)\cap Y|\le3\sqrt\varepsilon m$ from Lemma~\ref{lem:pretil} (iii) and $|W|\le2\beta n\le3\beta M m$.  By the choice of $\gamma$,
the chromatic number of the underlying graph induced by this common neighbourhood in $D$ is  at least $$\frac{((p-\varepsilon)^2-3\sqrt\varepsilon-3\beta M)m}{\gamma n}\geq \frac{((p-\varepsilon)^2-3\sqrt\varepsilon-3\beta M)m}{ \gamma Mm/(1-\eps)}\geq \kappa_\ell$$ so Lemma~\ref{lem:tree} gives a copy $Q$ of $\Cvec-s$, where $s$ is a source. Hence both $Q\cup\{x\}$ and $Q\cup\{z\}$ are copies of $\Cvec$: every edge of $J_i[U_i]$ is a {$(\Cvec,2\beta n,1)$-reachable} pair.

{Now let $W_0\subseteq V(D_1)$ with $|W_0|\le\beta n$.} For $x,y\in U_i$ with $xy\notin E(J_i)$, there is a vertex $z\in N_{J_i}(x)\cap N_{J_i}(y)\setminus W_0$ as $x,y\notin B_i$, $|W_0|\le\beta n$ and $(U_i,U_{f(i)})$ is $(3\varepsilon,d/2)$-regular directed pair. {Put $W_1:=W_0\cup\{y\}$ and use $(\Cvec,2\beta n,1)$-reachability of $xz$ to get a connector $Q_1$ avoiding $W_1$. Next put $W_2:=W_0\cup\{x\}\cup V(Q_1)$ and use $(\Cvec,2\beta n,1)$-reachability of $zy$ to get a connector $Q_2$ avoiding $W_2$. For sufficiently large $n$, both $|W_1|$ and $|W_2|$ are at most $2\beta n$. Hence $Q_1,Q_2,\{x,y,z\}$ are mutually disjoint in the required way. With $Q:=Q_1\cup Q_2\cup\{z\}$, the digraph $D_1[Q\cup\{x\}]$ has the two $\Cvec$-factors supplied by $Q_1\cup\{x\}$ and $Q_2\cup\{z\}$, while $D_1[Q\cup\{y\}]$ has those supplied by $Q_1\cup\{z\}$ and $Q_2\cup\{y\}$. Moreover $|Q|\le2\ell-1$.} Thus $U_i$ is $(\Cvec,\beta n,2)$-closed.
\end{proof}

Call two indexes $i,j$ \textbf{even-related} if there is an even-length walk in $R$ joining $i$ and $j$. {Since $R$ is connected (Lemma \ref{lem:connected}), it is easy to check that any two indexes $i,j$ are even-related if $R$ is non-bipartite and if $R$ is bipartite with $A_R,B_R$, then any two indexes $i,j$ are even-related, where $i,j\in A_R$ or $i,j\in B_R$, i.e., there are exactly two classes.}

We now merge these clusters $U_1,\ldots,U_k$ by even-length walk.

\begin{lemma}\label{lem:evenmerge}
Either $V(D_1)$ is $(\Cvec,\tau n,T)$-closed, or there is a bipartition
$V(D_1)=\mathcal A\cup\mathcal B$ such that $\mathcal A$ and $\mathcal B$ are $(\Cvec,4\tau n,T)$-closed. Moreover, if $V(D_1)=\mathcal A\cup\mathcal B$ such that $\mathcal A$ and $\mathcal B$ are $(\Cvec,4\tau n,T)$-closed, then there is a $(\Cvec,\rho)$-robust vector
$2\mathbf e_{\mathcal A}+(\ell-2)\mathbf e_{\mathcal B}$ with respect to the partition $(\mathcal A,\mathcal B)$.
\end{lemma}

\begin{proof}
By Lemma~\ref{lem:clusterclosed}, every cluster $U_i$ is $(\Cvec,\beta n,2)$-closed.
On the edge $ij$, Lemma \ref{lem:transferral} gives a $(\Cvec,\beta)$-robust vector  $\ei_i+(\ell-1)\ei_j$; on $jh$ it gives a $(\Cvec,\beta)$-robust vector  $(\ell-1)\ei_j+\ei_h$.  Their difference is $\mathbf e_i-\mathbf e_h$.  Hence Lemma \ref{lem:merge} merges $U_i$ and $U_h$.  Note that this step merges only the $(\Cvec,\beta n,2)$-closed clusters  $U_i$ and $U_h$.  Iterating this operation along even walks yields the usual parity decomposition. Since $R$ is connected by Lemma~\ref{lem:connected}, if $R$ is non-bipartite, all clusters lie in one even-walk class and $V(D_1)$ is $(\Cvec,\tau n,T)$-closed. Thus it remains to consider the case that $R$ is bipartite.

Let $(A_R,B_R)$ be the bipartition of $R$, let
$\mathcal A:=\bigcup_{i\in A_R}U_i$, $B:=\bigcup_{j\in B_R}U_j$,
and let
$\mathcal A^0:=\bigcup_{i\in A_R}V_i$,
$\mathcal B^0:=\bigcup_{j\in B_R}V_j$.
Then we assert that $\mathcal A$ and $\mathcal B$ are $(\Cvec,4\tau n,T)$-closed. Since both  $R[A_R]$ and $R[B_R]$ have no edge, $D'$ contains no arc inside $\mathcal A^0$ or inside $\mathcal B^0$. Consequently,
\begin{equation}\label{A0}
    \text{for every $x\in\mathcal A^0$,  there is
$d_D^+(x,\mathcal A^0)+d_D^-(x,\mathcal A^0)\le2(d+2\varepsilon)n,$}
\end{equation}

and the analogous inequality holds for $\mathcal B^0$. Let $\theta:=2d+4\varepsilon+2\sqrt\varepsilon.$
Also, for every $x\in\mathcal A$, $N_{D'[\mathcal A^0]}(x)=\emptyset$, so
$d_D^*(x)\le |\mathcal B|+\theta n$.
Similarly, for each vertex $y\in\mathcal B$, there is
$d_D^*(y)\le |\mathcal A|+\theta n$.
Since $|\mathcal A|+|\mathcal B|=|D_1|\ge(1-\sqrt\varepsilon)n$, it follows that $|\mathcal A|\geq (1-\sqrt\varepsilon)n/2$ or $|\mathcal B|\geq (1-\sqrt\varepsilon)n/2$. W.l.o.g., assume that $|\mathcal B|\geq |\mathcal A|$.  {Then $|\mathcal B|\geq (1-\sqrt\varepsilon)n/2>2(d+2\varepsilon)n$, which means that there are two non-adjacent vertices  $x,y\in\mathcal B$ in $D$.} The condition  $\sigore(D)\ge(\frac34+\mu)n$ gives
\[
\left(\frac34+\mu\right)n
\le d_D^*(x)+d_D^*(y)
\le2|\mathcal A|+2\theta n,
\]
so
$|\mathcal A|\ge\left(\frac38+\frac\mu2-\theta\right)n>2(d+2\varepsilon)n$, and hence $D[\mathcal A]$ contains a non-adjacent pair $x',y'$. Applying the same argument to $x',y'$ gives
$|\mathcal B|\ge\left(\frac38+\frac\mu2-\theta\right)n$.
Taking $d,\varepsilon\ll\mu^2$, we may therefore assume
\begin{equation}\label{eq:ore-side-bounds}
\left(\frac38+\frac{2\mu}{5}\right)n
\le |\mathcal A|,|\mathcal B|
\le\left(\frac58-\frac{2\mu}{5}\right)n.
\end{equation}

Let $W\subseteq V(D_1)$ with $|W|\le\rho n$. By symmetry between $P^+$ and $P^-$, suppose first that
$|\mathcal A\cap P^+|\ge\frac{|\mathcal A|}{2}$.
Define
\[
L:=\left\{x\in\mathcal A\cap P^+:d_D^*(x)<\left(\frac38+\frac\mu3\right)n\right\}.
\]
The condition  $\sigore(D)\ge(\frac34+\mu)n$  implies that $L$ is a complete in $UG(D)$: otherwise two distinct non-adjacent vertices $x,y$ of $L$ would have   $d_D^*(x)+d_D^*(y)<(3/4+\mu)n$, which contradicts $\sigore(D)\ge(\frac34+\mu)n$. Together with \eqref{A0}, we obtain that
$|L|\le2(d+2\varepsilon)n+1.$

By \eqref{eq:ore-side-bounds} and the hierarchy $d,\varepsilon,\rho,\gamma\ll\mu$, we have
$\left|(\mathcal A\cap P^+)\setminus(L\cup W)\right|>\alpha(D).$
Thus this set contains an arc, say $u\to v$. In particular,
$d_D^*(u),d_D^*(v)\ge\left(\frac38+\frac\mu3\right)n.$
Since $u,v\in P^+$ and $D'[\mathcal A^0]$ contains no arcs, after deleting $\mathcal F_E$,  we have
\[
d_{D_1}^+(u,\mathcal B)\ge d_D^*(u)-\theta n,
\qquad
d_{D_1}^+(v,\mathcal B)\ge d_D^*(v)-\theta n.
\]
Together with \eqref{eq:ore-side-bounds},
\begin{align*}
|N_{D_1}^+(u,\mathcal B)\cap N_{D_1}^+(v,\mathcal B)|
\ge d_{D_1}^+(u,\mathcal B)+d_{D_1}^+(v,\mathcal B)-|\mathcal B|\ge\left(\frac18+\frac{16\mu}{15}-2\theta\right)n.
\end{align*}

Since $\theta,\rho\ll\mu$, the set
$S:=\bigl(N_{D_1}^+(u,\mathcal B)\cap N_{D_1}^+(v,\mathcal B)\bigr)\setminus W$ has size at least $n/8$. Choose a source $s$ of $\Cvec$ and put $P_s:=\Cvec-s$. Since $\alpha(D[S])\le\alpha(D)$, Lemma~\ref{lem:tree} embeds a copy $Q$ of $P_s$ in $D[S]$. Choose a source $r$ of $P_s$ and let $q_r$ be its image in $Q$. Replace $q_r$ by $v$ and place $u$ at $s$. If $s$ and $r$ are adjacent in $\Cvec$, the required arc is precisely $u\to v$. Since every vertex of $S$ is a common out-neighbour of $u$ and $v$, it follows that
$(V(Q)\setminus\{q_r\})\cup\{u,v\}$
spans a $\Cvec$-copy with exactly two vertices in $\mathcal A$ and $\ell-2$ vertices in $\mathcal B$, avoiding $W$.

If instead $|\mathcal A\cap P^-|\ge|\mathcal A|/2$, apply the identical argument with in-neighbourhoods: choose the arc $u\to v$ from the corresponding good set, choose a sink $t$ of $\Cvec$, put $P_t:=\Cvec-t$, and choose a sink $r$ of $P_t$. Embed $P_t$ in a common in-neighbourhood, replace the image of $r$ by $u$, and place $v$ at $t$; if $r$ and $t$ are adjacent, the required arc is $u\to v$. This again gives a copy with profile $2\mathbf e_{\mathcal A}+(\ell-2)\mathbf e_{\mathcal B}$. Since $W$ was arbitrary, this vector is $(\Cvec,\rho)$-robust.
\end{proof}

\begin{corollary}\label{cor:globalclosed}
The digraph $D_1$ is $(\Cvec,\tau n,L')$-closed.
\end{corollary}

\begin{proof}
 Apply Lemma \ref{lem:evenmerge} for the digraph $D_1$, there is nothing to prove if $ V(D_1)$ already is $(\Cvec,\tau n,T)$-closed.  Hence, assume $V(D_1)$ has a bipartition $V(D_1)=\mathcal A\cup\mathcal B$ such that $\mathcal A$ and $\mathcal B$ are {$(\Cvec,4\tau n,T)$-closed} and there is a $(\Cvec,\rho)$-robust vector
 $2\mathbf e_{\mathcal A}+(\ell-2)\mathbf e_{\mathcal B}$ as $\rho \ll \nu$. Choose an edge $ij\in E(R)$ with $U_i\subseteq\mathcal A$ and $U_j\subseteq\mathcal B$ (such an edge exists because $R$ is connected).  {Recall that $(U_i,U_j)$ is a $(3\varepsilon,d/2)$-regular pair. Lemma~\ref{lem:transferral} with $\eps, d, M$, there is a $(\Cvec,\rho)$-robust vector}
 $\mathbf e_{\mathcal A}+(\ell-1)\mathbf e_{\mathcal B}$.
Their difference is exactly
 $\mathbf e_{\mathcal A}-\mathbf e_{\mathcal B}.$
{Apply Lemma~\ref{lem:merge} inside $D_1$ with $L:=T$ and $\beta:=4\tau$ to obtain that $D_1$ is $(\Cvec,2\tau |D_1|,L')$-closed. The fact  $|D_1|\ge n/2$ implies that   $2\tau|D_1|\ge\tau n$. Thus $D_1$ is $(\Cvec,2\tau |D_1|,L')$-closed.}
\end{proof}

Now we give the needed absorber lemma.

\begin{theorem}
[Ore Absorber]\label{cor:oreabsorber}
There is a set $A\subseteq V(D_1)$ with $|A|\le\eta n$ such that $D_1[A\cup U]$ has a $\Cvec$-factor whenever
$ U\subseteq V(D_1)\setminus A$ with
$ |U|\le\xi n$ and  $\ell\mid |A\cup U|$.
\end{theorem}
\begin{proof}
Let $n_1:=|D_1|$. By Corollary~\ref{cor:globalclosed}, $D_1$ is $(\Cvec,\tau n,L')$-closed, and hence in particular $(\Cvec,\tau n_1,T)$-closed. Apply Lemma~\ref{lem:absorber} to the oriented graph $D_1$ with parameters $(\eta,\tau,L')$. It yields a set $A\subseteq V(D_1)$ with $|A|\le\eta n_1\le\eta n$ which absorbs every $U\subseteq V(D_1)\setminus A$ with $|U|\le\xi_0 n_1$ and $\ell\mid |A\cup U|$. Since $\xi=\xi_0/2$ and $n_1\ge n/2$, we have $\xi n\le\xi_0n_1$.
\end{proof}
\subsection{Almost Covering Lemma}

Fix any subset $A\subseteq V(D_1)$ with $|A|\le\eta n$, and set $D_2:=D_1-A$ and $n_2:=|D_2|$. Let $r=n-n_2=|V(\mathcal F_E)|+|A|$. Since $r\le\sqrt\varepsilon n+\eta n\ll\mu n$, {for every two distinct non-adjacent vertices $x,y\in V(D_2)$, they are also non-adjacent in $D$, and hence}
\[
d_{D_2}^*(x)+d_{D_2}^*(y)\ge(3/4+\mu)n-2r\ge(3/4+\mu/2)n_2.
\]
Thus {$\sigore(D_2)\ge(3/4+\mu/2)n_2$}. Also $\alpha(D_2)\le\alpha(D)\le2\gamma n_2$. Define the dominant-direction partition for the residual digraph $D_2$ by $P_2^+:=\{v\in V(D_2):d_{D_2}^+(v)\ge d_{D_2}^-(v)\}$ and $P_2^-:=V(D_2)\setminus P_2^+$. Apply Lemma~\ref{lem:regularity} to $D_2$, refining {the partition formed by the nonempty members of $(P_2^+,P_2^-)$}. Let $V'_0,V'_1,\dots,V'_{k'}$ be the resulting partition, with common cluster size $m'$, pure digraph $D'_2$, and reduced digraph $R^0_2$. Set $R_2=UG(R^0_2)$. {By Lemma \ref{lem:connected}, $R_2$ is connected.}

For the graph $R_2$ on vertex set $[k]$, define the collection formed by pairs of index vectors associated with each edge of $R_2$:
\[
\mathcal{P}_\ell(R_2):=
\left\{
\mathbf e_i+(\ell-1)\mathbf e_j,\,
(\ell-1)\mathbf e_i+\mathbf e_j
\,\big|\, ij\in E(R_2)
\right\}.
\]
For every edge of $R_2$, the robust index-vector lemma yields a pair of $\tilde C$-copies realising the two index vectors in its associated pair, respectively.
Accordingly, we first obtain a fractional combination of the vectors from these vector pairs equal to $\mathbf 1_k$. We then greedily find an appropriate number of vertex-disjoint $\tilde C$-copies, which yields an almost-covering of the vertex set $V(D_2)$.

\begin{lemma}[Farkas' Lemma]\cite{farkas1902}
Let $\mathcal P=\{P_1,P_2,\dots,P_m\}\subseteq\mathbb R^k$ be a finite set of vectors. Exactly one of the following two statements holds.
\begin{enumerate}
    \item There exist nonnegative real scalars $\lambda_1,\lambda_2,\dots,\lambda_m\ge 0$ such that
    $\boldsymbol 1_k=\sum_{i=1}^m \lambda_i P_i.$

    \item There exists a vector $\boldsymbol y\in\mathbb R^k$ satisfying
    $\boldsymbol y^\top P_i \ge 0$ for all  $i=1,2,\dots,m$,
    and
    $\boldsymbol y^\top \boldsymbol 1_k<0.$
\end{enumerate}
\end{lemma}

\begin{lemma}\label{lem:profilehall}
There are coefficients $\lambda_P\ge0$, for each  $P\in\mathcal P_\ell(R_2)$,  such that
\[
\one_{k'}=\sum_{P\in\mathcal P_\ell(R_2)}\lambda_PP.
\]
\end{lemma}

\begin{proof}
Suppose not. By Farkas' lemma there is $\mathbf y=(y_1,\dots,y_{k'})\in\mathbb R^{k'}$ such that $\sum_i y_i<0$ and, for every edge $ij\in E(R_2)$,
\[
y_i+(\ell-1)y_j\ge0,\qquad (\ell-1)y_i+y_j\ge0.\tag{5}\label{eq:farkas}
\]
Set $S=\{i\in [k']:y_i<0\}$ and $a_i=-y_i$ for $i\in S$. Then $S$ is independent.

We claim that every nonempty independent set $I\subseteq V(R_2)$ satisfies
$(\ell-1)|N_{R_2}(I)|\ge |I|$.
If $|I|=1$, this follows from connectivity of $R_2$ and $\ell-1\ge2$. Suppose that $|I|\ge2$ and, to the contrary, $|I|>2|N_{R_2}(I)|$. Put $q:=|N_{R_2}(I)|$,
$X:=\bigcup_{i\in I}V'_i$, $Y:=\bigcup_{j\in N_{R_2}(I)}V'_j$.
Since $I$ is independent, all contributing to $w_i$ for $i\in I$ lies in $N_{R_2}(I)$; hence $w_i\le q$. Lemma~\ref{lem:weightedore}(ii) gives
$q\ge\frac\mu4k'$, so, for $n$ large,
$|Y|=qm'\ge\frac\mu5n_2$ and  $|X|=|I|m'>2|Y|$.  Moreover $D'_2[X]$ is empty. Thus every $x\in X$ has at most $2(d+2\varepsilon)n_2$ neighbours in $X$ in $UG(D_2)$. Since $|X|>2\mu n_2/5$ and $d,\varepsilon\ll\mu$, there are two distinct non-adjacent vertices $x,y\in X$. If $x\in V'_i\subseteq P^+_2$, then all out-neighbours of $ i$ in $R^0_2$ lies in $N_{R_2}(I)$, and therefore
\[
d_{D_2}^*(x)=d_{D_2}^+(x)\le |Y|+2(d+2\varepsilon)n_2.
\]
The same bound follows from in-degrees when $V'_i\subseteq P^-_2$. Hence $\sigore(D_2)\ge(3/4+\mu/2)n_2$ yields
\[
\left(\frac34+\mu/2\right)n_2
\le d_{D_2}^*(x)+d_{D_2}^*(y)
\le2|Y|+4(d+2\varepsilon)n_2,
\]
and so
$|Y|\ge\left(\frac38+\frac\mu4-2d-4\varepsilon\right)n_2>\frac{n_2}{3}.$
But $|X|>2|Y|$, whence $|X|+|Y|>3|Y|>n_2$, a contradiction. Thus $2|N_{R_2}(I)|\ge|I|$, which implies $(\ell-1)|N_{R_2}(I)|\ge |I|$ because $\ell-1\ge2$.

From \eqref{eq:farkas}, each $j\in N_{R_2}(S)$ satisfies
$y_j\ge(\ell-1)\max_{i\in S\cap N_{R_2}(j)}a_i$.
For $t\ge0$ put $S_t=\{i\in S:a_i>t\}$. The layer-cake identity and $(\ell-1)|N_{R_2}(I)|\ge |I|$ give
\[
\sum_{j\in N_{R_2}(S)}y_j\ge(\ell-1)\int_0^\infty |N_{R_2}(S_t)|\,dt
\ge\int_0^\infty |S_t|\,dt
=\sum_{i\in S}a_i=-\sum_{i\in S}y_i.
\]
Hence, $\sum_{j\in N_{R_2}(S)}y_j
 +\sum_{i\in S}y_i\geq 0. $
Note that, for each $j\notin S$, there is $y_j\geq 0$, so $\sum_{i=1}^{k'} y_i\ge0$, a contradiction.
This completes the proof.
\end{proof}

\begin{theorem}[Almost covering]\label{lem:almost}
The digraph $D_2$ admits a $\Cvec$-tiling leaving at most $\xi n$ vertices uncovered.
\end{theorem}
\begin{proof}
By Lemma~\ref{lem:profilehall}, there are coefficients $\lambda_P\ge0$, for each  $P\in\mathcal P_\ell(R_2)$,  such that
\[
\one_{k'}=\sum_{P\in\mathcal P_\ell(R_2)}\lambda_PP.
\]
For $P\in\mathcal P_\ell(R_2)$, set $N_P=\lfloor(1-\zeta)\lambda_Pm'\rfloor$. For every cluster  $V'_i$,
\[
 (1-\zeta)m'\geq
\sum_{\mathcal P_\ell(R_2)} c_i(P)N_P\ge(1-\zeta)m'-O(k'^2),
\tag{7}\label{eq:rounding}
\]
where $c_i(P)\in \{0,1,\ell-1\}$ is the $i$-th coordinate of $P$.

We process each index vector $P$ one by one, embedding $N_P$ copies of $\tilde C$ satisfying this index vector. The inequality \eqref{eq:rounding} ensures that after each iteration, every cluster retains at least $\zeta m'$ unused vertices. Since $k'\le M_a$ and $|V'_0|\le\varepsilon_a n_2\le n_2/2$, we obtain
\[
m'=\frac{n_2-|V'_0|}{k'}\ge\frac{n_2}{2M_a},
\qquad
\zeta m'\ge\frac{n_2}{M_a^\star}.
\]
Let $ij\in E(R'_2)$ be the support edge corresponding to the current index vector. Then the unused subsets, say $V''_i$ and $V''_j$, of $V'_i$ and $V'_j$ are both of size at least $\zeta m'$. By Lemma~\ref{lem:slicing}, $(V''_i,V''_j)$ form a $(\widehat\varepsilon_a,d_a/2)$-regular pair: the regularity parameter is at most $\max\{2\varepsilon_a,\varepsilon_a/\zeta\}=\widehat\varepsilon_a$, and the density is at least $d_a-\varepsilon_a\ge d_a/2$. Since  $|V''_i|,|V''_j|\geq n_2/M_a^\star$ and $\alpha(D_2)\le2\gamma n_2$, then Lemma~\ref{lem:transferral} to be applied to $D_2$ with parameters $(\widehat\varepsilon_a,d_a/2,M_a^\star)$ and $W=\emptyset$. It embeds the required two disjoint $\Cvec$-copies in $V''_i\cup V''_j $. Hence the greedy procedure is valid until all prescribed copies have been embedded.

At the end, \eqref{eq:rounding} leaves at most $\zeta m'+O(k'^2)$ vertices in each regularity cluster, in addition to $V'_0$. Since $k'$ is bounded independently of $n$,
\[
|V(D_2)\setminus V(\mathcal T)|\le \varepsilon_a n_2+\zeta n_2+O(1)<\xi n
\]
for sufficiently large $n$.
\end{proof}

\subsection{Completion of the Proof}

\begin{proof}[Proof of Theorem~\ref{thm:main}]
Fix $\ell,\Cvec$ and $\mu>0$.  Choose a small constant $\eta>0$ with $\eta\ll\mu$.  {Apply Lemma~\ref{lem:pretil}} to obtain the exceptional $\Cvec$-tiling $\mathcal F_E$  with $|V(\mathcal F_E)|\leq 2\ell \eps n$, covering the bad set $E_{bad}$. Also, denote $D_1:=D-\mathcal F_E$ to be the remaining digraph.

By Corollary~\ref{cor:oreabsorber}, $D_1$ contains an absorbing set $A$ with
$ |A|\le\eta n$
such that $D_1[A\cup U]$ has a $\Cvec$-factor whenever
$ U\subseteq V(D_1)\setminus A$ with
$ |U|\le\xi n$ and  $\ell\mid |A\cup U|$.
Set
 $D_2:=D_1-A$, $ n_2:=|D_2|$, and set
$r:=n-n_2{ {(=|V(\mathcal F_E)|+|A|)}}$.  The choice $\eps$ and $\eta$ yields  $r\le\mu n/10$. Further,  {for every two distinct non-adjacent vertices $x,y\in V(D_2)$},
\begin{align*}
 d_{D_2}^*(x)+d_{D_2}^*(y)
\ge\left(\frac34+\mu\right)n-2r\ge\left(\frac34+\frac\mu2\right)n_2.
\end{align*}
Also $\alpha (D_2)\le\alpha (D)$. {Apply Theorem~\ref{lem:almost}} to obtain a $\Cvec$-tiling $\mathcal T$ in $D_2$ whose uncovered set
$U:=V(D_2)\setminus V(\mathcal T)$
satisfies $|U|\le\xi n$.
 {Since $\ell\mid n$ and both $|V(\mathcal F_E)|$ and $|V(\mathcal T)|$ are multiples of $\ell$, we have $\ell\mid |A\cup U|$.} According to the absorption property of $A$, $D[A\cup U]$ contains a $\Cvec$-factor.  Together with $\mathcal T$ and the initial pre-tiling $\mathcal F_E$, this is a $\Cvec$-factor of $D$, which completes the proof.
\end{proof}

\section{Asymptotic sharpness}

 {We finish by showing that $3/4$ is the optimal  coefficient over the class of non-consistently-directed cycle orientations covered by Theorem~\ref{thm:main}; the obstruction   occurs for anti-directed cycles.}
Before proceeding further, we present the following example.
\begin{lemma} \label{lem:host}
For every integer $m\ge2$ there is an $m$-vertex tournament $R_m$ such that
\[
 \ds_{R_m}(v)\ge F(m):=\left\lfloor\frac{3m-2}{4}\right\rfloor
 \qquad(\text{for each }v\in V(R_m)).
\]
\end{lemma}

\begin{proof}
For every $r\ge 1$, there exists an $r$-vertex tournament $J_r$ such that
\[
d^+_{J_r}(v),d^-_{J_r}(v)\ge\lfloor(r-1)/2\rfloor
\]
for every vertex $v$. To construct $J_r$, take a regular tournament when $r$ is odd, and delete one vertex from a regular tournament on $r+1$ vertices when $r$ is even.

Set $a=\lfloor m/2\rfloor$ and $b=\lceil m/2\rceil$. Let $J_a$  and $J_b$ be two tournaments as defined above, and orient every edge from $A$ to $B$ to obtain a desired tournament. Then
\[
 \ds(u)\ge b+\left\lfloor\frac{a-1}{2}\right\rfloor\quad(\text{for each }u\in A),\
 \ds(v)\ge a+\left\lfloor\frac{b-1}{2}\right\rfloor\quad(\text{for each }v\in B),
\]
as desired.
\end{proof}

Let $Q$ be an oriented graph, let $R$ be a tournament disjoint from $Q$, and add a vertex $z$. Define $D(z,Q,R)$ by keeping $Q$ and $R$, orienting
 $z\to Q\to R,$
and there is no arc between $z$ and $R$.

\begin{lemma} \label{lem:root}
For $D(z,Q,R)$, the following hold.
\begin{enumerate}
 \item If $\Delta^+(Q)\le1$, then there is no $C_4^{\ad}$ containing the vertex $z$.
 \item If $s\ge3$ and the matching number of the underlying graph of $Q$ is at most $1$, then there is no $C_{2s}^{\ad}$ containing the vertex $z$.
\end{enumerate}
\end{lemma}

\begin{proof}
Since $d_D^-(z)=0$, the vertex $z$ must be a source in every anti-directed cycle containing it.
For $C_4^{\ad}$, its two sink-neighbours, say $z_1,z_2$, lie in $V(Q)$. So the common source, say $z'$, of $z_1$ and $z_2$ cannot lie in $R$, because there are no arcs from $R$ to $Q$. Hence  $z'\in V(Q)$ and $z'$ has two out-neighbours $z_1,z_2$ in $Q$, contradicting $\Delta^+(Q)\le1$.

Assume that $s\ge3$, write the sources as $x_1,\ldots,x_s$ and the sinks as $y_1,\ldots,y_s$, with
$x_i\to y_i$ and $x_{i+1}\to y_i$ (indices modulo $s$). If $x_1=z$, then $y_1,y_s\in V(Q)$. Analogously, their other source-neighbours $x_2,x_s$ also satisfy $x_2,x_s\in V(Q)$.
Thus
 $x_2\to y_1$ and $x_s\to y_s$
are two vertex-disjoint arcs of $Q$, a contradiction.
\end{proof}


\begin{proof}[Proof of Theorem \ref{thm:no-additive-intro}]
Let $Q$ be an oriented graph on $r$ vertices satisfying the conditions stated below. For all sufficiently large integers $n$ such that $2s\mid n$, set $m:=n-r-1$, let $R=R_m$ be the tournament given in Lemma~\ref{lem:host}, and define $D:=D(z,Q,R)$.
 Since $r$ is fixed, for all sufficiently large $n$ we have $r\le F(m)\le m$. Moreover
\[
 \ds_D(z)=r,\ \ds_D(u)\ge m\ (\text{for each }u\in Q),\
 \ds_D(v)\ge F(m)\ (\text{for each }v\in R).
\]
The five possible pair types $(z,Q),(z,R),(Q,Q),(Q,R),(R,R)$ show that
 $d_D^*(x)+d_D^*(y)\ge r+F(m)$ for every two distinct vertices $x,y$. Hence, in particular,
\begin{equation}\label{eq:master}
{\sigore(D)}\ge r+F(m) \ge r+\frac{3m-5}{4} =\frac34n+\frac14r-2.
\end{equation}
Moreover, since $Q\to R$, $z\to Q$, $z$ is nonadjacent to $R$, and $R$ is a tournament,
\begin{equation}\label{eq:alpha}
\alpha(D)=\max\{\alpha(Q),2\}.
\end{equation}

{Given the prescribed constant $C$, choose an integer $a\ge2$ so large that $\frac34a-2\ge C$ when $s=2$, and $\frac14a-\frac32\ge C$ when $s\ge3$.} Assume first that $s=2$.  Let $Q$ be the disjoint union of $a$ directed triangles. Then
$r=3a,$ $\alpha(Q)=a$,  $\Delta^+(Q)=1$.
By Lemma~\ref{lem:root}, the vertex $z$ is not contained in any $C_4^\text{ad}$, so $D$ admits no $C_4^\text{ad}$-factor. From \eqref{eq:alpha}, $\alpha(D)=a$, whereas \eqref{eq:master} gives
\[
 {\sigore(D)}\ge \frac34n+\frac34a-2.
\]

Next suppose $s\ge 3$.  Let $Q$ consist of one directed triangle together with $a-1$ isolated vertices. Then
$r=a+2$ and $\alpha(Q)=a$,
and $Q$ has matching number $1$. Lemma~\ref{lem:root} implies that $z$ belongs to no $C_{2s}^\text{ad}$, and hence no such factor exists. Again $\alpha(D)=a$, and \eqref{eq:master} produces
\[
 {\sigore(D)}\ge \frac34n+\frac14a-\frac32,
\]
as desired.
\end{proof}

\section{Conclusion and Remark}

In this paper, we establish Ore-type dominant-degree conditions in the Ramsey--Tur\'an setting for the existence of factors of arbitrary non-consistently-directed cycles in oriented graphs. We conclude with examples demonstrating that there exist orientations for which the Ore-type dominant-degree condition depends on the independence number yet the corresponding factor still fails to exist. It is therefore worthwhile to formulate analogous orientation-dependent coefficients for arbitrary non-consistently-directed cycles, so as to determine the optimal Ore-type dominant-degree conditions.

On the one hand, anti-directed cycle factors can serve as a natural starting point; namely, one can consider the following problem.

\begin{problem}\label{prob:main}
For each fixed $s\ge2$, determine the least constant $c_s$ for which there exists $K_s$ such that every sufficiently large $n$-vertex oriented graph $D$, with $2s\mid n$, satisfying
\[ \sigore(D) \ge\frac34n+c_s\alpha(D)+K_s
\]
has a $C_{2s}^{\ad}$-factor.
\end{problem}
Theorem~\ref{thm:no-additive-intro} gives
 $c_2\ge\frac34 $ and $c_s\ge\frac14\quad(s\ge3)$.

On the other hand, it would be interesting to establish analogous conclusions for general digraphs. Furthermore, one may investigate Ramsey--Tur\'an factor problems for arbitrary orientations of a fixed graph $H$.


\begin{thebibliography}{99}
\bibitem{AddarioBerry2013} L. Addario-Berry, F. Havet, C. Linhares Sales, B. Reed and S. Thomass\'e, \emph{Oriented trees in digraphs}, Discrete Math. \textbf{313} (2013), 967--974.
\bibitem{AlonShapira} N. Alon and A. Shapira, \emph{Testing subgraphs in directed graphs}, J. Comput. System Sci. \textbf{69} (2004), 354--382.
\bibitem{AlonYuster} N. Alon and R. Yuster, \emph{$H$-factors in dense graphs}, J. Combin. Theory Ser. B \textbf{66} (1996), 269--282.
\bibitem{BaloghLoMolla} J. Balogh, A. Lo and T. Molla, \emph{Transitive triangle tilings in oriented graphs}, J. Combin. Theory Ser. B \textbf{124} (2017), 64--87.
\bibitem{BaloghMollaSharifzadeh} J. Balogh, T. Molla and M. Sharifzadeh, \emph{Triangle factors of graphs without large independent sets and of weighted graphs}, Random Structures Algorithms \textbf{49} (2016), 669--693.
\bibitem{book} J. Bang-Jensen and G. Gutin, \emph{Digraphs: Theory, Algorithms and Applications}, 2nd ed., Springer, London, 2009.
\bibitem{ChangWeiYan2026} Y. Chang, S. Wei and J. Yan, \emph{An exact dominant degree condition for transitive tournament factors in digraphs}, arXiv:2608.10445, 2026.
\bibitem{ChenAntiDirected} M. Chen, \emph{Anti-directed cycle-factors in oriented graphs}, J. Graph Theory \textbf{113} (2026), 131--142.
\bibitem{ChenKouMaZhang2026} M. Chen, Y. Kou, X. Ma and S. Zhang, \emph{$TT_3$-factors in oriented graphs with low independence number}, Discrete Math. \textbf{349} (2026), 115229.
\bibitem{CHWY2024} M. Chen, J. Han, G. Wang and D. Yang, \emph{$H$-factors in graphs with small independence number}, J. Combin. Theory Ser. B \textbf{169} (2024), 373--405.
\bibitem{CHY2025} M. Chen, J. Han and D. Yang, \emph{Clique-factors in graphs with low $K_\ell$-independence number}, arXiv:2509.16851, 2025.
\bibitem{CKM2014} A. Czygrinow, H. A. Kierstead and T. Molla, \emph{On directed versions of the Corr\'adi--Hajnal corollary}, European J. Combin. \textbf{42} (2014), 1--14.
\bibitem{CorradiHajnal} K. Corr\'adi and A. Hajnal, \emph{On the maximal number of independent circuits in a graph}, Acta Math. Acad. Sci. Hungar. \textbf{14} (1963), 423--439.
\bibitem{DeBiasioLoMollaTreglown} L. DeBiasio, A. Lo, T. Molla and A. Treglown, \emph{Transitive tournament tilings in oriented graphs with large minimum total degree}, SIAM J. Discrete Math. \textbf{35} (2021), 250--266.
\bibitem{ErdosHajnalSosSzemeredi} P. Erd\H{o}s, A. Hajnal, V. T. S\'os and E. Szemer\'edi, \emph{More results on Ramsey--Tur\'an type problems}, Combinatorica \textbf{3} (1983), 69--81.
\bibitem{ErdosSos} P. Erd\H{o}s and V. T. S\'os, \emph{Some remarks on Ramsey's and Tur\'an's theorem}, in \emph{Combinatorial Theory and its Applications II}, North-Holland, 1970, 395--404.
\bibitem{farkas1902}
J. Farkas, \emph{\"Uber die Theorie der einfachen Ungleichungen},
{Journal f\"ur die reine und angewandte Mathematik},
\textbf{124} (1902),  1--27.

\bibitem{HajnalSzemeredi} A. Hajnal and E. Szemer\'edi, \emph{Proof of a conjecture of P. Erd\H{o}s}, in \emph{Combinatorial Theory and its Applications II}, North-Holland, 1970, 601--623.
\bibitem{HMWY2024} J. Han, P. Morris, G. Wang and D. Yang, \emph{A Ramsey--Tur\'an theory for tilings in graphs}, Random Structures Algorithms \textbf{64} (2024), 94--124.
\bibitem{HNY2026} X. He, X. Nie and D. Yang, \emph{Transversal tilings in $k$-partite graphs without large holes}, arXiv:2602.10578, 2026.
\bibitem{Kelly2011} L. Kelly, \emph{Arbitrary orientations of Hamilton cycles in oriented graphs}, Electron. J. Combin. \textbf{18} (2011), Paper 186, 25 pp.
\bibitem{KellyKuhnOsthus} L. Kelly, D. K\"uhn and D. Osthus, \emph{A Dirac-type result on Hamilton cycles in oriented graphs}, Combin. Probab. Comput. \textbf{17} (2008), 689--709.
\bibitem{KnierimSu} C. Knierim and P. Su, \emph{$K_r$-factors in graphs with low independence number}, J. Combin. Theory Ser. B \textbf{148} (2021), 60--83.
\bibitem{KomlosSimonovits} J. Koml\'os and M. Simonovits, \emph{Szemer\'edi's regularity lemma and its applications in graph theory}, in \emph{Combinatorics, Paul Erd\H{o}s is Eighty}, Vol. 2, Bolyai Soc. Math. Stud. \textbf{2}, 1996, 295--352.
\bibitem{KuhnOsthus} D. K\"uhn and D. Osthus, \emph{The minimum degree threshold for perfect graph packings}, Combinatorica \textbf{29} (2009), 65--107.
\bibitem{Lo2025} A. Lo, \emph{From finding a spanning subgraph $H$ to an $H$-factor}, arXiv:2509.18832, 2025.
\bibitem{LoMarkstrom} A. Lo and K. Markstr\"om, \emph{$F$-factors in hypergraphs via absorption}, Graphs Combin. \textbf{31} (2015), 679--712.
\bibitem{MollaTreglown2026} T. Molla and A. Treglown, \emph{Cycle tilings and $H$-factors in directed graphs}, arXiv:2602.13737, 2026.
\bibitem{SimonovitsSos} M. Simonovits and V. T. S\'os, \emph{Ramsey--Tur\'an theory}, Discrete Math. \textbf{229} (2001), 293--340.
\bibitem{Treglown} A. Treglown, \emph{On directed versions of the Hajnal--Szemer\'edi theorem}, Combin. Probab. Comput. \textbf{24} (2015), 873--928.
\bibitem{WWY2026} Z. Wang, Z. Wang and J. Yan, \emph{Ramsey--Tur\'an type problem for perfect transitive triangle tilings in digraphs}, arXiv:2606.14161, 2026.
\bibitem{Yuster2003} R. Yuster, \emph{Tiling transitive tournaments and their blow-ups}, Order \textbf{20} (2003), 121--133.
\bibitem{ZHOUGAO} J. Zhou and Y. Gao, \emph{Ramsey--Tur\'an anti-directed cycle factors in oriented graphs}, arXiv:2608.08591, 2026.
\end{thebibliography}
\end{document}